\documentclass[11pt,oneside]{amsart}
\usepackage[width=5.5in,height=8.8in]{geometry}
\usepackage{float}
\usepackage{enumerate}
\title[The Complexity of Arc-Connectedness Relation in the Plane]{The Complexity of Arc-Connectedness Relation in the Plane}
\author{YUSUF UYAR}
\address{Department of Mathematics, Middle East Technical University, 06800, Ankara, Turkey.}
\email{yuyar@metu.edu.tr}
\usepackage{hyperref}
\usepackage{amsmath, amssymb, amsthm}
\newtheorem{theorem}{Theorem}
\newtheorem{lemma}[theorem]{Lemma}
\newtheorem{proposition}[theorem]{Proposition}

\newcommand{\N}{\mathbb{N}}
\newcommand{\I}{\mathbb{I}}
\newcommand{\R}{\mathbb{R}}
\newcommand{\Z}{\mathbb{Z}}
\newcommand{\J}{\mathcal{J}}
\begin{document}
\sloppy
\begin{abstract}
In this paper, we show that the arc-connectedness equivalence relation on a Polish subspace of the real plane is an essentially hyperfinite Borel equivalence relation. This result provides the optimal upper bound for such a Borel equivalence relation.
\end{abstract}
\maketitle
\section{Introduction}
For a Polish space $X$, the arc-connectedness equivalence relation $\approx _X$ on $X$ is an analytic relation because we have $$x\approx_Xy \text{ iff there exists } f \in C(\I,X) \text{ such that }f(0)=x \text{ and }f(1)=y$$ 
for all $x,y\in X$, where $C(\I, X)$ is the Polish space of continuous functions\footnote{ We note that the function $f$ in this definition is generally called a path rather than an arc, where the latter term is used to mean a homeomorphic image of $\I$. However, path-connectedness coincides with arc-connectedness in Hausdorff spaces (see \cite[Theorem 2.11]{illanes2025}).} from the unit interval $\I$ to $X$. 

That any analytic equivalence relation on the Cantor space $2^\N$ can be realized as the arc-connectedness equivalence relation on a compact subset of $\R^3$ immediately follows from Theorem 4.1 of \cite{becker1998}, in which Becker asked how complex the arc-connectedness equivalence relation can be for compact subsets of $\R^2$. Then Becker and Pol showed that for any Polish subspace $X\subseteq \R^2$, the relation $\approx _X$ is Borel if and only if all arc components of $X$ are Borel \cite{becker2001}.  In fact, it has been proven that this equivalence relation is Borel for any Polish subspace of $\R^2$ \cite{saintraymond2024}.

In this paper, we will show that the arc-connectedness relation on any Polish subspace of $\R^2$ is an essentially hyperfinite Borel equivalence relation, and therefore is bounded by the eventual equality relation $ E_0$ on $2^\N$. 

This upper bound turns out to be the best possible. One can actually embed the Cantor space into the plane suitably to see that $E_0$ is Borel reducible to the arc-connectedness relation on the Knaster continuum (also known as the bucket handle). Indeed, the same result is immediate from \cite{solecki2002} as composants of this continuum coincide with its arc components. The reader may consult \cite[\textsection 48, V-VI]{kuratowski1968} for the construction and related continuum-theoretic properties of the Knaster continuum.

\textbf{Related Work.} We refer the reader to \url{https://arxiv.org/abs/2601.00601}, an extended version of this document written in collaboration with additional authors, for a more comprehensive analysis of connectedness relations on Polish spaces.

\section{Borel Equivalence Relations}

In this section, we recall the necessary definitions and results from the theory of Borel equivalence relations. Our reference for this section is \cite{dougherty1994}. 

Let $X$ be a standard Borel space, and let $E\subseteq X\times X$ be a Borel equivalence relation. Then $E$ is called \textit{finite} (respectively \textit{countable}) if every equivalence class of it is finite (respectively countable). It is called \textit{hyperfinite} if $E=\bigcup _{n\in \N}E_n$, where $\{E_n\}_{n\in \N}$ is an increasing sequence of finite Borel equivalence relations on $X$.

Let $F$ be a Borel equivalence relation on a standard Borel space $Y$. Then $E$ is said to be \textit{Borel reducible} to $F$, written $E\leq _BF$, if there exists a Borel function $f$, called a \textit{Borel reduction}, from $X$ to $Y$ such that $xEy$ if and only if $f(x)Ff(y)$ for all $x,y\in X$. We consider $E$ not to be more complex than $F$, provided that such a Borel reduction exists.

A Borel equivalence relation is called \textit{essentially hyperfinite} if it is Borel reducible to a hyperfinite Borel equivalence relation. For any countable set $\boldsymbol{\Omega}$ with $|\boldsymbol{\Omega}|>1$, the eventual equality relation $E_0(\boldsymbol{\Omega})$ given by $$\alpha E_0(\boldsymbol{\Omega})\beta \text{ iff }\exists n\in \N\ \forall m\geq n\ \alpha _m=\beta _m$$ on the product space $\boldsymbol{\Omega}^\N$ is a universal hyperfinite Borel equivalence relation in the sense that any hyperfinite Borel equivalence relation is Borel reducible to it. By the transitivity of Borel reducibility, any essentially hyperfinite Borel equivalence relation is Borel reducible to $E_0(\boldsymbol{\Omega})$. The equivalence relation $E_0(2)$ is usually denoted by $E_0$.

We know that $E$ is hyperfinite if and only if it is the orbit equivalence of a Borel $\Z$-action on $X$. One can easily modify the proof of this fact to obtain the following slightly more general result (for example, see \cite[Theorem 6.6]{kechris2004} and the remark below it).
\begin{theorem}
\label{graph}
   Suppose that there exists a Borel graph on $X$ with degree at most $2$ such that the connected components of the graph are exactly $E$-classes. Then $E$ is hyperfinite.
\end{theorem}

We recall that a subset $A\subseteq X$ satisfying $$A=\bigcup _{x\in A}[x]_E$$ is said to be $E$-\textit{invariant}. The following lemma helps us to consider different parts of arc components separately.

\begin{lemma}
\label{lemma2}
   Let $\{B_n:n\in \N\}$ be a countable collection of $E$-invariant Borel subsets of $X$ whose union is $X$. If each $E|_{B_n}$ is essentially hyperfinite, then so is $E$.
\end{lemma}
\begin{proof}
    We set $A_0=B_0$ and recursively define $A_n=B_n\setminus \bigcup _{m<n}B_m$ for all $n\geq 1$. One can check by induction that $\{A_n\}_{n\in \N}$ is an $E$-invariant Borel partition of $X$.
    
    For each $n\in \N$, let $f_n:B_n\rightarrow \N^\N$ be a Borel reduction from $E|_{B_n}$ to $E_0(\N)$. Then the map $f:X\rightarrow \N^\N$ defined by $$f(x)=\left( f_k(x)_0,k,f_k(x)_1,k,f_k(x)_2,k,... \right)$$ where $x\in A_k$ is a Borel map because each $f|_{A_n}$ is Borel. Moreover, it is also a Borel reduction from $E$ to $E_0(\boldsymbol{\N})$, showing that $E$ is essentially hyperfinite.
\end{proof}

\section{Upper Bound}
In this section, we adopt some definitions, notations, and well-known facts from \cite{saintraymond2024} to prove the upper bound. Let $H\subseteq \R^2$ be a Polish subspace. We endow the set $\mathcal{K}(H)$ of all compact subsets of $H$ with the Vietoris topology, which is a Polish space. The set $\J(H)\subseteq \mathcal{K}(H)$ of all arcs in $H$ and the map $e:\J(H)\rightarrow \mathcal{K}(H)$ that assigns to each arc the set of its endpoints are Borel (see \cite{becker2001}). We define $e_1(J)=\min \left(e(J)\right)$ and $e_2(J)=\max \left(e(J)\right)$ for all $J\in \J(H)$ with respect to the lexicographic order on $\R^2$, which is Borel. The equivalence relation $\approx _{\J(H)}$ defined on $\J(H)$ of being in the same arc component is given by $$\{(I,J)\in \J(H)\times \J(H):e_1(I)\approx _He_1(J)\}$$ and so it is Borel. 

For any $A\subseteq H$, we denote the restriction ${\approx _H}|_A$ by $\approx _ A$. Likewise, ${\approx _{\J(H)}}|_{\mathcal{A}}$ will be denoted by $\approx _{\mathcal{A}}$ for any $\mathcal{A}\subseteq \J(H)$.

Let $C_x$ denote the arc component of $x\in H$. The partition
\begin{align*}
H_0=&\{x\in H: C_x\text{ is a continuous injective image of a real interval}\}\\
H_1=&\{ x\in H: C_x\text{ contains a simple triod} \}\\
H_2=& \{x\in H: C_x\text{ is a singleton}\}
\end{align*}
of the set $H$ was proven to be a Borel partition \cite{saintraymond2024}. 

We shall work with non-triodic, non-degenerate arc components. Fix two rational numbers $p<q$. Let $L,R$ be the vertical lines $x=p$ and $x=q$, respectively. We define the Borel set $$\J_{p,q}=\{J\in \J(H):J\cap L=\{e_1(J)\},\ J\cap R=\{e_2(J)\},\ e_1(J)\in H_0\}$$ of all arcs in $H_0$ that intersect the lines $L$ and $R$ at their endpoints only. We will show that $\J_{p,q}$ "chooses a chain of arcs" from the corresponding arc components of $H_0$.

\textbf{Observations.} Let $C$ be an arc component in $H_0$, and let $\J_{p,q}|_C$ denote elements of $\J_{p,q}$ that lie in $C$. Say $\varphi:K\rightarrow C$ is a continuous injective parametrization where $K\subseteq \R$ is an interval.

It is easy to show that $\varphi^{-1}( J)\subseteq K$ is a closed interval for any arc $J\in \J_{p,q}|_C$\footnote{with the only exception that if $C$ is a simple closed curve, then the inverse image of only one element of $\J_{p,q}|_C$ may violate this, in which case the subsequent observations and propositions can be easily modified}. Any distinct arcs $J,J'\in \J_{p,q}|_C$ can intersect only at their endpoints, and therefore $\varphi^{-1}(J)$ and $\varphi^{-1}(J')$ are two intervals that possibly intersect at one of their endpoints (see \cite[Lemma 7]{saintraymond2024}). This allows us to linearly order the elements of $\J _{p,q}|_C$. The following proposition shows that this ordering is, in fact, "chain-like".

\begin{proposition}
\label{proposition3}
For any distinct $J,J'\in \J_{p,q}|_C$, there can be only finitely many elements of $\{ \varphi^{-1}\left( I \right):I\in \J_{p,q} |_C\}$ between $\varphi ^{-1}( J)$ and $\varphi ^{-1}(J')$.
\end{proposition}
\begin{proof} Let  $J,J'\in \J_{p,q}|_C$ be distinct arcs. Say $[a,b]=\varphi ^{-1}( J)$ and $[a',b']=\varphi ^{-1}(J')$. Assume without loss of generality that $b\leq a'$. If $b=a'$, then we are done. So, suppose that $b<a'$. Assume to the contrary that the set $$\mathcal{T}=\{\varphi ^{-1}(I)\subseteq [b,a']:I\in \J_{p,q}|_C\}$$ is infinite. Then let $\{[x_n,y_n]\}_{n\in \N}$ be a decreasing (or increasing) sequence of intervals in $\mathcal{T}$. Set $L_H=H\cap L$ and $R_H=R\cap H$.
     
The sequence $\{x_n\}_{n\in \N}$ has distinct elements and is a subset of $\varphi^{-1}(L_H\cup R_H)$. Without loss of generality, we may assume that $\{x_n\}_{n\in \N}\cap \varphi^{-1}(L_H)$ is infinite. Say $\{x_{k_n}\}_{n\in \N}\subseteq \varphi^{-1}(L_H)$ where $\{x_{k_n}\}_{n\in \N}$ is a subsequence of $\{x_n\}_{n\in \N}$. Then $\{y_{k_n}\}_{n\in \N}\subseteq \varphi^{-1}(R_H)$. 

On the other hand, we have $\displaystyle \lim _{n\rightarrow \infty}x_{k_n}=\lim _{n\rightarrow \infty}y_{k_n}$ because $\{[x_{k_n},y_{k_n}]\}_{n\in \N}$ is a decreasing (or increasing) sequence of intervals in $[b,a']$. It follows from the continuity of $\varphi$ that $\displaystyle \lim _{n\rightarrow \infty}\varphi(x_{k_n})=\lim _{n\rightarrow \infty}\varphi(y_{k_n})$. However, this is not possible because $\{\varphi(x_{k_n})\}_{n\in  \N}\subseteq L$ and $\{\varphi(y_{k_n})\}_{n\in  \N}\subseteq R$. Thus, $\mathcal{T}$ is finite. 
\end{proof}

Let us call $J,J'\in \J_{p,q}|_C$ \textit{consecutive} if there is no interval in $$\{\varphi^{-1}(I):I\in \J_{p,q}|_C\}$$ that lies between $\varphi^{-1}(J)$ and $\varphi^{-1}(J')$. The above proposition shows that the parametrization $\varphi$ defines a canonical graph on $\J_{p,q}|_C$ of being consecutive whose degree is at most $2$. We can define this graph on the entire set of arcs $\J_{p,q}$ without using parametrizations by characterizing the relation of being consecutive as follows.

\begin{proposition}
\label{proposition4}
The arcs $J$ and $J'$ are consecutive if and only if they intersect at an endpoint of them, or there exists an arc $\Gamma\subseteq H$ that satisfies $e(\Gamma)\subseteq e(J)\cup e(J')$ and does not intersect both $L$ and $R$.
\end{proposition}
\begin{proof}
    $(\rightarrow)$ Assume that $J$ and $J'$ are disjoint consecutive arcs. Then $\varphi([b,a'])$ is the desired arc $\Gamma$. $(\leftarrow)$ If $J$ and $J'$ intersect at an endpoint, then we are done. So assume that there exists an arc $\Gamma\subseteq H$ that does not intersect both $L$ and $R$ such that $e(\Gamma)\subseteq e(J)\cup e(J')$. But then $[b,a']\subseteq \varphi ^{-1}(\Gamma)$ and so $\varphi([b,a'])$ cannot intersect both $L$ and $R$. So, there is no element of $\{\varphi^{-1}(I):I\in \J_{p,q}|_C\}$ between $\varphi^{-1}(J)$ and $\varphi^{-1}(J')$.
\end{proof}

Using the proposition above, we define pairs of disjoint and intersecting consecutive arcs of $\J_{p,q}$ separately as follows. First, consider the Borel set $$\{ (J,J',\Gamma)\in \J_{p,q}^2\times \J(H):e(\Gamma)\subseteq e(J)\cup e(J')\text{ and } [\Gamma\cap L= \emptyset \text{ or } \Gamma\cap R= \emptyset]\}$$
whose projection to $\J_{p,q}^2$ is the set, call $\mathcal{D}_1$, of pairs of disjoint consecutive arcs. Each $(J,J')$-section of the Borel set is finite because $\Gamma$ is determined by the endpoints of $J$ and $J'$. It follows from the Lusin-Novikov Theorem \cite[Theorem 18.10]{kechris1995} that $\mathcal{D}_1$ is Borel. Second, we consider the Borel set $$\mathcal{D}_2=\{ (J,J')\in \J_{p,q}^2:J\neq J' \text{ and }[e_1(J)=e_1(J') \text{ or }e_2(J)=e_2(J')]\}$$ of all consecutive intersecting arcs. Then Propositions \ref{proposition3} and \ref{proposition4} together imply that $\mathcal{D}=\mathcal{D}_1\cup \mathcal{D}_2$ is a Borel graph on $\J_{p,q}$ of degree at most $2$ whose connected components are exactly $\approx _{\J_{p,q}}$-classes.

We define the horizontal analogue of $\J_{p,q}$ as follows. We set $\mathcal{I}_{p,q}$ to be the set of arcs $J$ contained in $H_0$ such that the intersection of $J$ with the horizontal line $y=p$ is the singleton containing an endpoint of $J$ and the intersection of $J$ with the line $y=q$ is the singleton containing the other endpoint.

\begin{lemma}
\label{lemma5}
    $\approx _{\J_{p,q}}$ and  $\approx _{\mathcal{I}_{p,q}}$ are hyperfinite.
\end{lemma}
\begin{proof}
We observed that the set $\mathcal{D}$ satisfies the hypothesis of Theorem \ref{graph} for $\approx _{\J_{p,q}}$. So, $\approx _{\J_{p,q}}$ is hyperfinite. The proof that $\approx _{\mathcal{I}_{p,q}}$ is hyperfinite is symmetric.
 \end{proof}

It is a direct consequence of the following theorem that the arc-connectedness relation on a Polish subspace of $\R^2$ can be no more complex than $E_0$.

\begin{theorem}
    Let $H\subseteq \R^2$ be a Polish subspace. Then $\approx _H$ is essentially hyperfinite. 
\end{theorem}

\begin{proof}
Being hyperfinite, $\approx _{\J_{p,q}}$ is a countable Borel equivalence relation. Then the Borel set
$$\mathcal{M}_{p,q}=\{(x,J)\in H_0\times \J_{p,q}:x\approx _He_1(J)\}$$ has countable sections. So, the set $\pi(\mathcal{M}_{p,q})$ of all points in $H_0$ that are arc-connected to an arc in $\J_{p,q}$ is Borel, and $\mathcal{M}_{p,q}$ has a Borel uniformization by the Lusin-Novikov Theorem. This uniformization is clearly a Borel reduction from $\approx _{\pi  (\mathcal{M}_{p,q})}$ to $\approx _{\J_{p,q}}$. Then by Lemma \ref{lemma5}, $\approx _{\pi  (\mathcal{M}_{p,q})}$ is essentially hyperfinite. Similarly, $\approx _{\pi  (\mathcal{N}_{p,q})}$ is also essentially hyperfinite where $$\mathcal{N}_{p,q}=\{(x,J)\in H_0\times \mathcal{I}_{p,q}:x\approx _He_1(J)\}\ .$$

It is evident from their definitions that $\pi (\mathcal{M}_{p,q})$ and $\pi (\mathcal{N}_{p,q})$ are $\approx _H$-invariant. 

For any distinct elements $x,y\in H_0$ with $\pi (x)<\pi (y)$ connected by an arc $J$, we can choose rationals $r,s$ such that $\pi (x)<r<s<\pi  (y)$. Then it is obvious that $J$ contains an element of $\J_{r,s}$ so that $x,y\in \pi  (\mathcal{M}_{r,s})$. But if $\pi  (x)=\pi (y)$, we can similarly choose two horizontal lines $y=r$ and $y=s$ between the points $x$ and $y$ so that $x,y\in \pi  (\mathcal{N}_{r,s})$. These arguments show that sets of the form $\pi  (\mathcal{M}_{p,q})$ and $\pi  (\mathcal{N}_{p,q})$  cover all $H_0$.

Due to Moore's Theorem \cite{pittman1970}, $H_1$ has countably many $\approx _H$-classes, and so it is Borel reducible to the identity relation on $\N$. Lastly, $\approx _{H_2}$ is already the identity relation on $H_2$. Clearly, $\approx _{H_1}$ and $\approx _{H_2}$ are essentially hyperfinite. Then $$\{\pi (\mathcal{M}_{p,q}):p<q \text{ are rationals }\}\cup \{\pi (\mathcal{N}_{p,q}):p<q \text{ are rationals }\}\cup \{H_1,H_2\}$$ is a countable collection as in Lemma \ref{lemma2}. Thus, $\approx _H$ is essentially hyperfinite.
\end{proof}

\bibliographystyle{amsalpha}
\bibliography{refs}

\end{document}